\documentclass[11pt]{amsart}

\usepackage{amsmath,amsthm}
\usepackage{amsfonts,amsrefs}
\usepackage{verbatim}
\usepackage{amssymb}
\usepackage{txfonts}
\usepackage{amscd}
\usepackage{bbm}
\usepackage{hyperref}
\usepackage{mathrsfs}
\usepackage{fouriernc}

\usepackage{tikz-cd}
\usepackage{gauss}
\usepackage{cleveref}

\newtheorem{theorem}{Theorem}

\newtheorem{lemma}[theorem]{Lemma}

\newtheorem{remark}[theorem]{Remark}

\newcommand{\diag}{\operatorname{diag}}

\newcommand{\cardS}{{\vert S\vert}}

\author{Piotr Mizerka}
\author{Piotr W. Nowak}
\address{Institute of Mathematics of the Polish Academy of Sciences, \'{S}niadeckich 8, 00-656, Warsaw, Poland}
\email{pnowak@impan.pl}
\email{pmizerka@impan.pl}

\begin{document}

\title{On order units in the augmentation ideal}

\begin{abstract}

We study order units in the real group ring and the augmentation ideal, as well as in matrix algebras. 
We identify an infinite family of order units in the powers of the augmentation
ideal, that includes the Laplacian, and show that these order units are naturally obtained via cohomological operations from 
more simpler diagonal order units in matrix algebras.
\end{abstract}
\maketitle

Algebraic methods have become useful in recent years in proving the existence of spectral gaps for 
group Laplacians, more generally cohomological Laplacians, and some operators related to them. This started with 
Ozawa's characterization of Kazhdan's property (T) in terms of sums of squares \cite{ozawa1}, which then 
led to proofs of this property for some classes of groups, including $\operatorname{Aut}(F_n)$, the automorphism
group of the free group on $n$ generators, for $n\ge 5$ \cite{kno, kkn}. It has also allowed to provide new characterizations of vanishing of higher cohomology with coefficients in unitary representations \cite{bader-nowak}. Recently, Ozawa  \cite{ozawa2} also interpreted Shalom's property $H_T$ in this setting.

In this work we are interested in the algebraic structure of the augmentation ideal and the order induced by the cone of sums of hermitian squares. We suppose that all groups occuring here are finitely generated. Ozawa \cite{ozawa1} showed that the group Laplacian $\Delta$ is an order unit in $I[G]$, the augmentation ideal of the group $G$. This fact was crucial for establishing a characterization of property $(T)$ in terms of an algebraic spectral gap for the Laplacian in the group ring $\mathbb{R}G$. In \cite{ozawa2} he also distinguished another element,  $\square$ (see definition below|), and showed in particular that it is an order unit in $I[G]$ for groups with finite abelianization. 

Here we put both the Laplacian $\Delta$ and the element $\square$ in a single framework. 
We define a family of elements $\square_n\in I[G]$ and show that they are order units in $I[G]$ for groups $G$ with finite abelianization. 
There are two essential ingredients in the proof of our main result, \Cref{theorem:main}. The first one is the property that, for any order unit $u$ in $I^{2n}[G]$, the diagonal matrix $\text{diag}(u)$ with $u$ in each entry of the diagonal, is an order unit in matrices over $I^{2n}[G]$. 
The second is the fact that for a group $G$ with finite abelianization generated by  a finite set $S$, a certain positive map $D:\mathbb{M}_{\cardS\times \cardS}(I[G])\to I[G]$ is surjective. 

An important advantage of this approach is that it provides a deeper explanation of why the Laplacian and the element $\square$ are both order units 
in the appropriate powers of the augmentation ideal, as they are constructed as images of natural order units in a matrix algebra under a positive map. 

\subsection*{Acknowledgements} 
We are grateful to Marek Kaluba for helpful comments.

Both authors were supported by the Maestro 13 National Science Center (Poland) grant "Analysis on Groups" 2021/42/A/ST1/00306.

\section{An infinite family of order units}\label{section:theorem_statement}
For any group $G$, one can define the \emph{real group ring} $\mathbb{R}G$ as the ring consisting of finitely supported functions $G\rightarrow\mathbb{R}$ with pointwise addition and convolution multiplication. It is convenient to write each element $\xi\in\mathbb{R}G$ as $\xi=\sum_g\xi(g)g$, indicating that $g\mapsto \xi(g)$ is the function $\xi$ defines. The map $G\rightarrow G$, $g\mapsto g^{-1}$ defines then the involution $*$ on $\mathbb{R}G$. We also consider the \emph{augmentation ideal} $I[G]$ which is the kernel of the augmentation map $\omega:\mathbb{R}G\rightarrow\mathbb{R}$. The $n$-th \emph{augmentation power} $I^n[G]\subseteq \mathbb{R}G$ is the span of $n$-fold products of elements of the
augmentation ideal $I[G]$. Since $I[G]$ is spanned by elements of the form $1-g$ for all $g\in G$, it is easy to see that
$I^n[G]$ is spanned by elements of the form $\prod_{i=1}^n (1-g_i)$, where $g_i\in G$. 

For any $k\geq 1$, the $*$-involution structure of $\mathbb{R}G$ (resp. $I[G]$) endows the matrices $\mathbb{M}_{k\times k}(\mathbb{R}G)$ (resp. $\mathbb{M}_{k\times k}(I[G])$) with the $*$-algebra structure, the $*$-involution being the composition of the $*$-involution on $\mathbb{R}G$ (resp. $I[G]$) and matrix transposition. 

Let $\mathcal{A}$ be a *-algebra. The \emph{positive cone} of \emph{hermitian squares} $\Sigma^2\mathcal{A}$ is the 
set of finite sums of the form $\sum_{i=1}^l a_i^*a_i$, where $a_i\in \mathcal{A}$. 
Let $V\subseteq \mathcal{A}$ be a subspace. An element $u\in V$ is an \emph{order unit} in $V$ if for every 
$v=v^*\in V$ there exists $R_v\ge 0$ such that $v+R_vu\in V\cap \Sigma^2\mathcal{A}$. The algebras
we consider here are of the form $\mathbb{M}_{k\times k}(\mathbb{R}G)$, and they have the property that 
the identity matrix is always an order unit in $\mathcal{A}$.

Let $G$ be a group with a finite generating set $S$. Then $d=[1-s]_{s\in S} \in \mathbb{M}_{\vert S\vert\times 1}(I[G])$ is the matrix of the $0$-codifferential map $\mathbb{R}G\to \oplus_{s\in S}I[G]$. The map $D:\mathbb{M}_{\cardS\times \cardS}(I[G])\to I[G]$ is defined by 
$$
D(\xi)=d^*\xi d.
$$

\subsection*{The family $\square_n$}
We will now introduce the family of elements of the agumentation ideal that will be the main focus of this article. 
For each $k\geq 1$, denote by $\text{diag}_k(\xi)$ the diagonal $k\times k$ matrix with each diagonal entry equal to $\xi\in\mathbb{R}G$ and 
we put $$\text{diag}(\square_{n-1})=\text{diag}_{|S|}(\square_{n-1}).$$ 
Define $$\square_0=1,$$ 
and for any $n\geq 1$ let
$$
\square_n=D\left(\text{diag}(\square_{n-1})\right)=\sum_{s_1,\ldots,s_n\in S}(1-s_n)^*\ldots(1-s_1)^*(1-s_1)\ldots(1-s_n).
$$
Note that 
$$\square_1=\Delta \text{\ \ \ \ and \ \ \ \ }\square_2=\square,$$ 
where $\square$ is as defined in \cite{ozawa2}, up to a normalizing constant (actually, $\square_2=4\square$; since scalar multiplication does not change the positivity type, we can assume $\square_2=\square$). 

We will prove the following 

\begin{theorem}\label{theorem:main}
	Suppose $G$ has finite abelianization. Then, for each $n\geq 1$, $\square_n$ is an order unit in $I[G]$ and $\text{diag}(\square_n)$ is an order unit in $\mathbb{M}_{|S|\times |S|}(I[G])$.
\end{theorem}

\section{Diagonal order units}\label{section:diagonal_ous}
Let $G$ be a group with a finite symmetric generating set $S$. Suppose that $n\geq 1$ and $u$ is an order unit in $I^{2n}[G]$. Let $k\geq 1$. We will 
prove that $\text{diag}_k(u)$ is an order unit in $\mathbb{M}_{k\times k}(I^{2n}[G])$.

For $g\in G$ and $s_1,\ldots,s_n,t_1,\ldots,t_n\in S$, denote by $\alpha_u$ the product $(1-u_1)\ldots(1-u_n)$, for $u=s,t$. Define
\begin{align*}
	E_{s,t}(\pm g)=\begin{gmatrix}[b]
		0&\pm\alpha_s g\alpha_t\\
		\pm\alpha_t^*g^{-1}\alpha_s^*&0
	\end{gmatrix}
\end{align*}
and 
\begin{align*}
	\square_{s,t}=\begin{gmatrix}[b]
		\alpha_s\alpha_s^*&0\\
		0&\alpha_t^*\alpha_t
	\end{gmatrix}.
\end{align*}
The following lemma is a crucial part in proving the general case:
\begin{lemma}\label{lemma:diag_square_order_unit_2_2}
	The two matrices $E_{s,t}(\pm g)+\square_{s,t}$ are sums of hermitian squares.
\end{lemma}
\begin{proof}
	This follows from the decomposition below:
\begin{align*}
		E_{s,t}(\pm g)+\square_{s,t}=\begin{gmatrix}[b]
			\pm\alpha_sg\\
			\alpha_t^*
		\end{gmatrix}\begin{gmatrix}[b]
			\pm g^{-1}\alpha_s^*&\alpha_t
		\end{gmatrix}.
	\end{align*}
\end{proof}
\noindent Applying the lemma above, we get the general statement:
\begin{lemma}\label{lemma:diag_ou_general}
	The matrix $\text{diag}_k(u)$ is an order unit in $\mathbb{M}_{k\times k}(I^{2n}[G])$.
\end{lemma}
\begin{proof}
	The case $k=1$ is obvious, since $u$ is an order unit in $I^{2n}[G]$. Let $k=2$ and suppose $M\in\mathbb{M}_{k\times k}(I^{2n}[G])$. Suppose first that $M$ consits of the diagonal part only. Since $u$ is an order unit in $I^{2n}[G]$ by our assumption, it follows that we can add a sufficient amount of $\diag_k(u)$ to $M$ to make it a sum of hermitian squares. It suffices therefore to prove the assertion for the non-diagonal part and we can suppose that the diagonal part of $M$ vanishes. Since $I[G]$ is generated by $1-s$, $s\in S$ as a left, as well as a right $\mathbb{R}G$-module, it follows that $I^{n}[G]$ is generated by the elements $\alpha_s$ as a left (right) $\mathbb{R}G$-module. Thus, $M$ can be expressed as a finite linear combination of matrices $E_{s,t}(\pm g)$ with positive coefficients:
	\begin{align*}
	M=\sum_{s,t,\pm g}\lambda_{s,t,\pm g}E_{s,t}(\pm g),\quad \lambda_{s,t,\pm g}>0.
	\end{align*}
It follows by \Cref{lemma:diag_square_order_unit_2_2} that $M'=M+\sum_{s,t,\pm g}\lambda_{s,t,\pm g}\square_{s,t}$ is a sum of hermitian squares. Since $S=S^{-1}$, we can add to $M'$ appropriate diagonal entries of the form $\lambda_s\alpha_s^*\alpha_s$ or $\lambda_s\alpha_s\alpha_s^*$, $\lambda_s>0$, to get $M''=M+\lambda\text{diag}_k(\square_n)$ for some $\lambda\geq 0$. Obviously, $M''$ is a sum of hermitian squares. On the other hand, since $u$ is an order unit in $I^{2n}[G]$, it follows that for a suitable $\lambda'>0$ the diference $\lambda'u-\square_n$ is a sum of hermitian squares. Thus, 
$$
M+\lambda\lambda'\text{diag}_k(u)=M''+\lambda\text{diag}_k(\lambda'u-\square_n)
$$
is a sum of hermitian squares.

The proof for the general case is just a repetition of the arguments of the general case proof of \cite[Proposition 3.2]{kaluba-mizerka-nowak}.
\end{proof}

\section{Surjectivity of $D$}\label{section:D_surjectivity}
We will now examine the map $D$ and prove that for groups $G$ with finite abelianization it is surjective. 

For a real vector space $V$ denote by $V'$ its dual, the space of all linear functionals on $V$.
For the real group ring $\mathbb{R}G$ the dual space $\mathbb{R}G'$ can be identified with  the space of all functions $\{ f:G\to \mathbb{R}\}$.
Denote the translation actions of $G$ on $\varphi\in \mathbb{R}G'$ by $g\cdot \varphi\cdot h(\eta)=\varphi(g^{-1}\cdot\eta\cdot h^{-1})$. This structure 
can be extended linearly to 
a bimodule structure on $\mathbb{R}G'$ over $\mathbb{R}G$.

As the augmentation ideal is a subspace $I[G]\subseteq \mathbb{R}G$, the dual $I[G]'$ is a quotient of $\mathbb{R}G'$. More precisely, 
the dual of the augmentation map, $\omega':\mathbb{R}\to \mathbb{R}G'$ is the inclusion of $\mathbb{R}$ as constant functions and $I[G]' = \mathbb{R}G'/\operatorname{const}$. We will thus view an element of $I[G]'$ as the equivalence class of functions on $G$ that differ by constant functions.

\begin{lemma}\label{theorem: D surjective}
Let $G$ be a finitely generated group generated by a finite set $S$. 
The following conditions are equivalent:
\begin{enumerate}
\item The augmentation ideal $I[G]$ is idempotent, \label{item: main theorem augmentation powers stable}
\item The map $\iota_*: H^1(G,M) \to H^1(G,N)$ is injective for any inclusion $\iota:M\hookrightarrow N$ of the trivial $G$-module $M$ into an $\mathbb{R}G$-module $N$,\label{item: main theorem injectivity in cohomology}
\item The map $D:\mathbb{M}_{\cardS\times \cardS}(I[G])\to I[G]$ is surjective. \label{item: main theorem D surjective}
\end{enumerate}
\end{lemma}

\begin{proof}
We first prove that \eqref{item: main theorem augmentation powers stable} implies \eqref{item: main theorem injectivity in cohomology}. Let $[z]\in H^1(G,M)$ be represented by a cocycle 
$z:G\to M$, such that $\iota_*([z])=0$ in $H^1(G,N)$. That is, 
for each $g\in G$ there exists $n\in N$ such that $(\iota\circ z)(g)=(1-g)n$ for every $g\in G$. Since $I^2[G]= I[G]$ and $I^2[G]$ is spanned by the elements of the form $(1-g)(1-h)$, 
we have that for every $g\in G$ the element $1-g$ can be expressed as a finite linear combination
$1-g = \sum \alpha_i (1-g_i)(1-h_i)$. Therefore, 
\begin{align*}
(1-g)n &= \sum \alpha_i(1-g_i)(1-h_i)n= \sum \alpha_i (1-g_i)\iota(z(h_i))=0,
\end{align*}
since $\iota$ is the inclusion of the trivial module. Since $\iota$ is also injective, this means that $z(g)=0$ for any $g\in G$.

To show that \eqref{item: main theorem injectivity in cohomology} implies \eqref{item: main theorem D surjective} consider the dual map $$D': \mathbb{R}G' \to \mathbb{M}_{\vert S\vert \times\vert S\vert}(I[G])'.$$ 
For $\xi \in \mathbb{M}_{\vert S\vert \times \vert S\vert }(I[G])$ we have 
\begin{align*}
(D'\varphi) (\xi) &= \varphi\left( \sum_{s\in S} (1-s)^* (\xi d)_s \right)
= \sum_{s\in S} (1-s)\cdot\varphi \left( \sum_{t\in S} \xi_{st} (1-t) \right)\\
&=\sum_{s,t\in S} (1-s)\cdot \varphi \cdot (1-t)^* (\xi_{st}).
\end{align*}

\noindent Assume that $\varphi\in \mathbb{R}G'$ belongs to the kernel of $D'$. Since $\xi_{st}$ can be any element from $I[G]$, we have for every $s,t\in S$ that
$$(1-s)\cdot \varphi_{\text{aug}} \cdot (1-t)^*=0,$$ 
where $\varphi_{\text{aug}}:G\to\mathbb{R}$, sends $g$ to $\varphi(1-g)$. Let $\psi_s=(1-s)\cdot\varphi_{\text{aug}} \in\mathbb{R}G'$. The condition 
$$\psi_s\cdot(1-t)^*=0,$$
 implies that for every $s\in S$ the element $\psi_s\in \mathbb{R}G'$ is a constant function on $G$.

Now observe that the map $\psi: G\to \mathbb{R}G'$ defined by $g\mapsto \psi_g=\varphi_{\text{aug}}-g\cdot \varphi_{\text{aug}}$, is a  1-coboundary for the left $G$-module 
$\mathbb{R}G'$.
It follows by the cocycle property that $\psi$ is in fact a cocycle $G\to \mathbb{R}$. By assumption, the map $\iota_*:H^1(G,\mathbb{R})\to H^1(G,\mathbb{R}G')$, induced
by the inclusion of coefficients $\mathbb{R}\subseteq \mathbb{R}G'$ as constant functions, is injective. 
Therefore, since the cocycle $\psi:G\to \mathbb{R}G'$ is a coboundary, it is also a coboundary as a 
cocycle $\psi:G\to \mathbb{R}$.
This however means that $\psi$ is identically zero, as the action of $G$ on $\mathbb{R}$ is trivial. 
Thus for every $s\in S$ we obtain
$$
\psi_s=\varphi_{\text{aug}}-s\cdot\varphi_{\text{aug}} =0,
$$
and consequently, $\varphi_{\text{aug}}\in \mathbb{R}G'$ is a constant function on $G$. It is straightforward that $\varphi$ has to be constant as well.

It follows that the map $D':I[G]'\to \mathbb{M}_{\vert S\vert \times\vert S\vert}(I[G])'$ is injective. Indeed, assume the contrary, then there exist two classes $[\varphi_1]\neq [\varphi_2]$ whose difference maps to 0 under $D'$. Choose any two representatives $\varphi_1$ and $\varphi_2$ of these classes, then the difference
$\varphi_1-\varphi_2$ is not a constant function and is in the kernel of the map $D': \mathbb{R}G'\to \mathbb{M}_{\vert S\vert \times \vert S\vert }(I[G])'$.

Finally, to derive \eqref{item: main theorem augmentation powers stable} from \eqref{item: main theorem D surjective}. Let $\xi\in I[G]$ and observe that since $\xi=d^*md$ for some $m\in \mathbb{M}_{\vert S\vert \times\vert S\vert}(I[G])$,
we have 
$$\xi = \sum_{s,t\in S} (1-s)^* m_{s,t} (1-t).$$
Since $m_{s,t}(1-t)\in I[G]$ we see that this implies $\xi$ is in fact an element of $I^2[G]$, giving the inclusion $I[G]\subseteq I^2[G]$. 
\end{proof}
 
Examples for which the conditions 
\eqref{item: main theorem augmentation powers stable}-\eqref{item: main theorem D surjective} of Theorem \ref{theorem: D surjective}
are not satisfied (that is, $G$ has infinite abelianization) can be given using results of Chen \cite{chen1,chen2} and include 
torsion-free nilpotent groups. The particular example of the Heisenberg group was also considered by Ozawa in \cite{ozawa2}.

Augmentation quotients $I^n[G]/I^{n+1}[G]$ in the context of group rings over general rings were also studied by Stallings \cite{stallings} and Quillen \cite{quillen}, 
who also considered their direct sum as a graded ring associated to the group ring.
In the situation we consider here the augmentation quotients and the associated graded ring are trivial. 
Augmentation powers $I^n[G]$ are also used to define dimension subgroups $D_n(G)=\left\{g\in G: 1-g \in I^n[G]\right\}$.
The identification of the dimension subgroups is a classical problem.

\begin{remark}\normalfont\label{remark: alternative}
As pointed out by the referee, the implication \eqref{item: main theorem augmentation powers stable} $\Longrightarrow$ \eqref{item: main theorem D surjective}
can be proven directly. 
Suppose $I[G]=I^2[G]$. Then $D$ is surjective.
Pick $\xi\in I[G]$. Since $I[G]=I^2[G]$, we have $I[G]=I^3[G]$ as well. Thus, we can express $\xi$ as the following finite sum:
	\begin{align*}
		\xi=\sum_{g,h,k}\lambda_{g,h,k}(1-g)(1-h)(1-k).
	\end{align*}
	As the augmentation ideal is generated by $1-s$, $s\in S$ as a left, as well as a right $\mathbb{R}G$-module, we can further write:
	\begin{align*}
		\xi=\sum_{g,h,k}\lambda_{g,h,k}\sum_{s\in S}(1-s)^*\alpha_{g,s}(1-h)\sum_{t\in S}\beta_{k,t}(1-t)
		=\sum_{s,t\in S}(1-s)^*\xi_{s,t}(1-t),
	\end{align*}
	where $\xi_{s,t}=\sum_{g,h,k}\lambda_{g,h,k}\alpha_{g,s}(1-h)\beta_{k,t}\in I[G]$. The observation that $$
	D\left(\left[\xi_{s,t}\right]_{s,t\in S}\right)=\xi
	$$ concludes then the proof.
\end{remark}

\section{Proof of \Cref{theorem:main} and final remarks}
We are now in the position to prove \Cref{theorem:main} . Recall that we have a group $G$ possessing finite abelinization and $n\geq 1$. Recall also, as noted at the beginning of \cref{section:D_surjectivity}, that the assumption that $G$ has finite abelianization is equivalent to stabilization of augmentation powers.
\begin{proof}[Proof of \Cref{theorem:main}]
	We prove the assertion by a simple induction. By \cite[Lemma 2]{ozawa1}, we know that $\square_1=\Delta$ is an order unit in $I[G]=I^2[G]$. It follows by \Cref{lemma:diag_ou_general} that $\text{diag}(\square_1)$ is an order unit in $\mathbb{M}_{|S|\times |S|}(I^2[G])$. Since $G$ has finite abelianization, we conclude that $\text{diag}(\square_1)$ is an order unit in $\mathbb{M}_{|S|\times |S|}(I[G])$. This proves the case $n=1$.
	
	Suppose the statement of \Cref{theorem:main} holds for some $n\geq 1$. Take any $\eta\in I[G]$ such that $\eta=\eta^*$. It follows by the surjectivity of $D$ that there exists $\xi\in\mathbb{M}_{|S|\times |S|}(I[G])$ such that $\eta=d^*\xi d$. Since $\eta=\eta^*$, we have
	\begin{align*}
		\eta = d^*\left(\dfrac{\xi+\xi^*}{2}\right)d + d^*\left(\dfrac{\xi-\xi^*}{2}\right)d= d^*\left(\dfrac{\xi+\xi^*}{2}\right)d.
	\end{align*}
By the inductive assumption, the matrix $\text{diag}(\square_n)$ is an order unit in $\mathbb{M}_{|S|\times |S|}(I[G])$. There exists therefore some $\lambda\geq0$ such that
$$
\dfrac{\xi+\xi^*}{2}+\lambda \cdot\text{diag}(\square_n) \in \mathbb{M}_{|S|\times |S|}(I[G])\cap\Sigma^2\mathbb{M}_{|S|\times |S|}(\mathbb{R}G).
$$
Since composing with $d^*$ on the left and $d$ on the right preserves the property of being a sum of hermitian squares, we obtain 
$$d^*\left(\dfrac{\xi+\xi^*}{2}\right)d+\lambda \cdot d^*\text{diag}(\square_n)d=\eta+ \lambda\cdot \square_{n+1} \in I[G]\cap\Sigma^2\mathbb{R}G.$$
\end{proof}

\begin{remark}\normalfont
	We can in fact extend \Cref{theorem:main} and provide an alternative proof that $\Delta$ is an order unit in $I[G]$ for groups with finite abelianiaztion. 
	
	This can be done as follows. Note that the map $D$ is surjective when extended to $\mathbb{M}_{|S|\times|S|}(\mathbb{R}G)$. On the other hand, the identity matrix $I$ is an order unit in $\mathbb{M}_{|S|\times|S|}(\mathbb{R}G)$, as noted in \cref{section:theorem_statement}. Thus, the proof of \Cref{theorem:main} would apply as well for that case. Indeed, for some $\lambda\geq 0$,
	$$
	\dfrac{\xi+\xi^*}{2}+\lambda I \in \Sigma^2\mathbb{M}_{\vert S\vert \times\vert S\vert}(\mathbb{R}G),
	$$
	and by the surjectivity of $D$, we get
	$$d^*\left(\dfrac{\xi+\xi^*}{2}\right)d+\lambda d^*d=\eta+ \lambda \Delta\in I[G]\cap\Sigma^2\mathbb{R}G.$$

	Even though this requires an additional assumption of $D$ being surjective (that is $G$ possessing finite abelianization), it provides 
	a conceptual explanation of this role of the Laplacian, showing that $\Delta$ is in that case the image of the identity 
	element in another ring under a positive map.
\end{remark}
\begin{remark}\normalfont
The argument from \Cref{remark: alternative} can be easily generalized to show that the map $D:\mathbb{M}_{\vert S\vert \times\vert S\vert}(I^{2n}[G])\rightarrow I^{2n+2}[G]$ is surjective for any $n\geq 1$, regardless of the property of $G$ possessing a finite abelianization. On the other hand, $\text{diag}(u)$ is an order unit in $\mathbb{M}_{\vert S\vert \times\vert S\vert}(I^{2n}[G])$, provided $u$ is an order unit in $I^{2n}[G]$, see \Cref{lemma:diag_ou_general}.Applying the same arguments as in the proof of \Cref{theorem:main} one can show then that if $u$ is an order unit in $I^{2n_0}[G]$ for some $n_0\geq 1$, then 
	$$
	u_n=\sum_{s_1,\ldots,s_n\in S}(1-s_n)^*\ldots(1-s_1)^*u(1-s_1)\ldots(1-s_n)
	$$ 
	is an order unit in $I^{2(n_0+n)}[G]$. This confirms in particular the observation of Ozawa that $\square$ is an order unit in $I^4[G]$.
\end{remark}
\begin{remark}\normalfont
In addition to $\Delta=D(I)$ and $\square=D(\text{diag}(\Delta))$, other distinguished elements of $I[G]$ have an explicit preimage in $\mathbb{M}_{\cardS\times \cardS}(\mathbb{R}G)$.

\begin{enumerate}
\item $\Delta^2=d^*dd^*d= D(dd^*)$, where $dd^*$ is the matrix with $(s,t)$-entry given by $(1-s)(1-t)^*$;
\item $\operatorname{Sq}=D (\widetilde{\operatorname{Sq}})$, where $\widetilde{\operatorname{Sq}}$ is the diagonal 
matrix with $(1-s)(1-s)^*$ in the $(s,s)$-entry and 0 otherwise.
\end{enumerate}

\end{remark}

\end{document}